\documentclass[11pt]{amsart}

\usepackage{amssymb}
\usepackage[margin=1in]{geometry}
\usepackage{nameref,hyperref}

\newcommand{\C}{{\mathbb C}}
\newcommand{\D}{{\mathbb D}}
\newcommand{\N}{{\mathbb N}}

\newcommand{\R}{{\mathbb R}}

\newcommand{\cA}{\mathcal{A}}
\newcommand{\cB}{\mathcal{B}}

\newcommand{\Hdim}{\mathcal{H}}

\newcommand{\inv}{^{-1}}
\newcommand{\subeq}{\subseteq}
\newcommand{\supeq}{\supseteq}
\newcommand{\ovl}{\overline}
\newcommand{\del}{\delta}
\newcommand{\bslash}{\backslash}

\DeclareMathOperator{\diam}{diam}
\DeclareMathOperator{\length}{length}

\DeclareMathOperator{\id}{id}
\theoremstyle{plain}
\newtheorem{theorem}{Theorem}
\newtheorem{proposition}[theorem]{Proposition}
\newtheorem{lemma}[theorem]{Lemma}
\newtheorem{corollary}[theorem]{Corollary}

\theoremstyle{definition}

\newtheorem{claim}{Claim}
\theoremstyle{remark}

\begin{document}
\bibliographystyle{amsalpha}
\title[Quasisymmetric uniformization]{Quantitative quasisymmetric uniformization of compact surfaces}
\author{Lukas Geyer}
\author{Kevin Wildrick}
\address{Department of Mathematical Sciences, Montana State
  University, Bozeman, MT 59717}
\email{geyer@montana.edu}
\email{kevin.wildrick@montana.edu}
\subjclass[2010]{Primary 30C65; Secondary 30C62, 51F99}

\maketitle

\begin{abstract}
  Bonk and Kleiner showed that any metric sphere which is Ahlfors
  2-regular and linearly locally contractible is quasisymmetrically
  equivalent to the standard sphere, in a quantitative way. We
  extend this result to arbitrary metric compact orientable
  surfaces.
\end{abstract}

\section{Introduction and statement of results}
Through isothermal coordinates, every Riemannian metric on a compact
orientable surface $S$ determines a Riemann surface structure on
$S$. By the classical uniformization theorem, $S$ carries a
conformally equivalent Riemannian metric of constant curvature $1$ (in
the case of a sphere), $0$ (for a torus), or $-1$ (for higher genus
surfaces). Hence, the original Riemannian metric on $S$ can be
\emph{conformally deformed} to a metric of constant curvature.

The purpose of this note is to extend the above discussion to certain
classes of possibly non-smooth distances on compact orientable
surfaces. In this setting we have a metric, but no smooth structure,
and the appropriate category replacing the class of conformal mappings
is the class of \emph{quasisymmetric mappings}.

In a metric space $X$ we will denote the distance between $a$ and $b$
by $|a-b|_X$, or $|a-b|$ if the metric space $X$ is clear from the
context.  Let $X$ and $Y$ be metric spaces. An embedding $\phi \colon
X \to Y$, i.e., a homeomorphism from $X$ onto its image $\phi(X)$, is
a \emph{quasisymmetric embedding} if there is a homeomorphism $\eta
\colon [0,\infty) \to [0,\infty)$ such that for all triples of
distinct points $a$, $b$, and $c$ in $X$,
\begin{equation*}
  \frac{|\phi(a)-\phi(b)|_Y}{|\phi(a)-\phi(c)|_Y} \leq \eta
  \left(\frac{|a-b|_X}{|a-c|_X}\right),
\end{equation*}
The homeomorphism $\eta$ is called a \emph{distortion function} for
the mapping $\phi$. If a quasisymmetric embedding $\phi$ has a
distortion function $\eta$, it will be called an $\eta$-quasisymmetric
embedding. The inverse of an $\eta$-quasisymmetric map is
$\eta'$-quasisymmetric with $\eta'(t) = 1/\eta^{-1}(1/t)$. A proof of
this fact can be found in the book of Heinonen
\cite[Proposition~10.6]{Heinonen2001}, which also serves as an
excellent introduction to the theory of quasisymmetric mappings on
metric spaces. A quasisymmetric mapping distorts \emph{relative}
distances by a controlled amount. While it may distort distances, it
must do so more-or-less isotropically.

A motivating example arises from Cannon's Conjecture in geometric
group theory. The boundary at infinity of a Gromov hyperbolic group
carries a natural family of visual metrics, any two of which are
quasisymmetrically equivalent. For this reason, one might say that the
metric on the boundary of a Gromov hyperbolic group is defined only up
to quasisymmetry. Cannon's Conjecture can be phrased as follows: If the boundary $\partial_\infty
G$ of a Gromov hyperbolic group $G$ is homeomorphic to the sphere
$\mathbb{S}^2$, then each visual metric on $\partial_\infty G$ is
quasisymmetrically equivalent to the standard metric on
$\mathbb{S}^2$. If Cannon's conjecture is true, then a Gromov hyperbolic group whose boundary is homeomorphic to $\mathbb{S}^2$ acts discretely and
co-compactly by isometries on 3-dimensional hyperbolic space. See
\cite{Bonk2006} and the references therein for a more detailed discussion of this
problem.

The following quasisymmetric uniformization theorem of Bonk and
Kleiner \cite{BonkKleiner2002} led to significant progress on Cannon's
Conjecture:

\begin{theorem}[Bonk-Kleiner]\label{sphere} Let $(X,d)$ be an Ahlfors
  $2$-regular metric space that is homeomorphic to
  $\mathbb{S}^2$. Then $(X,d)$ is quasisymmetrically equivalent to
  $\mathbb{S}^2$ equipped with the smooth metric of constant curvature
  $1$ if and only if $(X,d)$ is linearly locally contractible.
\end{theorem}

A metric space $X$ is \emph{linearly locally contractible}, if there
is a constant $\Lambda \geq 1$ such that for each point $x \in X$ and
radius $0<r<(\diam X)/\Lambda$, the ball $B(x,r)$ is contractible
inside the dilated ball $B(x,\Lambda r)$. The metric space $X$ is
\emph{Ahlfors $2$-regular} if there is a constant $K \geq 1$ such that
for each $x \in X$ and radius $0<r<2\diam X$, the two-dimensional
Hausdorff measure $\Hdim^2$ satisfies
\begin{equation*}
  \frac{r^2}{K} \leq \Hdim^2(B(x,r)) \leq Kr^2.
\end{equation*}

Theorem~\ref{sphere} is quantitative in the sense that the
quasisymmetric map can be chosen to have a distortion depending only
on the \emph{data} of $X$, i.e., the constants $\Lambda$ and $K$
appearing in the linear local contractibility and Ahlfors
$2$-regularity conditions.

The linear local contractibility condition is a quantitative
quasisymmetric invariant, i.e., if $f \colon X \to Y$ is
$\eta$-quasisymmetric and $X$ is linearly locally contractible with
constant $\Lambda$, then $Y$ is linearly locally contractible with a
constant that depends only on $\Lambda$ and $\eta$.  In contrast,
Ahlfors $2$-regularity is not a quasisymmetric invariant.

Given a Gromov-hyperbolic group $G$ with $\partial_\infty G$
homeomorphic to $\mathbb{S}^2$, each visual metric on $\partial_\infty
G$ is linearly locally contractible. However, it is not known whether
there always is an Ahlfors $2$-regular visual metric on
$\partial_{\infty}G$. Simple examples show that Theorem~\ref{sphere}
fails without the assumption of Ahlfors $2$-regularity.

Problems outside of geometric group theory, such as the search for an
intrinsic characterization of $\R^2$ up to bi-Lipschitz equivalence
(note that Ahlfors $2$-regularity \emph{is} a bi-Lipschitz invariant),
led to the development of versions of Theorem~\ref{sphere} for
non-compact simply connected surfaces \cite{Wildrick2008} and for a
large class of planar domains \cite{MerenkovWildrick2013}, as well as
a local version \cite{Wildrick2010}. In this work, we provide a
version of Theorem~\ref{sphere} that applies to all orientable compact
surfaces, including those of higher genus. This is part of an ongoing
program to complete the classification of Ahlfors $2$-regular and
linearly locally contractible metric surfaces up to quasisymmetry.

\begin{theorem}\label{thm:main intro} Let $(X,d)$ be a
  metric compact orientable surface. Assume that $(X,d)$ is Ahlfors
  $2$-regular and linearly locally contractible. Then there exists a
  Riemannian metric $\hat{d}$ of constant curvature $1$, $0$, or $-1$
  on $X$ such that the identity map $\id:(X,d) \to (X,\hat{d})$ is
  quasisymmetric with a distortion function depending only on the
  data of $X$.
\end{theorem} 

Here the \emph{data} of $X$ consists of the topological genus of $X$,
and the constants in the Ahlfors $2$-regularity and linear local
contractibility conditions.

Except for the statement of dependence of the distortion function only
on the data of $X$, Theorem \ref{thm:main intro} follows immediately
from the local uniformization theorem of \cite{Wildrick2010} and an
elementary local-to-global result for quasisymmetric mappings due to
Tukia and V\"ais\"al\"a \cite[Theorem 2.23]{TukiaVaisala1980}. While a
quantitative proof is significantly more involved, it provides much
more information about the ``space'' of quasisymmetric structures on
surfaces. For example, let us suppose that $(X,d)$ satisfies the
assumptions of Theorem~\ref{thm:main intro} and is homeomorphic to a
torus. All smooth Riemannian metrics of constant curvature $0$ on the
torus are quasisymmetrically equivalent, but not with a uniform
distortion function. In essence, the quantitative
Theorem~\ref{thm:main intro} allows us to assign to $(X,d)$ a point in
the Riemann moduli space of the torus that is ``roughly
optimal''. Finding a smooth metric on the torus that is
quasisymmetrically equivalent to $(X,d)$ with ``minimal''
quasisymmetric distortion seems to be both a difficult and ill-defined
question, but one of obvious interest.

Our proof of Theorem~\ref{thm:main intro} proceeds as follows. Let
$(X,d)$ be a metric surface as given in the theorem. By the local
uniformization result of \cite{Wildrick2010}, $(X,d)$ has an atlas of
uniformly quasisymmetric mappings. Our first step is to create a
compatible conformal atlas, giving $(X,d)$ the structure of a Riemann
surface. This step can be thought of as creating ``quasi-isothermal''
coordinates; while the chart transitions are conformal, the chart
mappings themselves are only quasisymmetries.  The classical
uniformization theorem then provides a globally defined conformal
homeomorphism $F \colon X \to Y$ to a Riemann surface that is the
quotient of the sphere, plane, or disk by an appropriate group of
M\"obius transformations. The Riemann surface $Y$ inherits a
Riemannian metric $d_Y$ of constant curvature $1$, $0$, or $-1$,
respectively. Although conformal maps are locally quasisymmetric, this
does not immediately give us global information about the metric
properties of the uniformizing map $F \colon (X,d) \to (Y,d_Y)$. We
show that the uniformizing mapping $F$ is in fact globally
quasisymmetric with a distortion function that depends only on the
data of $X$. The key idea in this step is a type of Harnack
inequality. The claim in the theorem then follows by taking $\hat{d}$
to be the pull-back of the metric $d_Y$ under $F$.

\section{Background and preliminary results}
\subsection{Notation}
In a metric space $X$, we will denote the distance between points $x$
and $y$ of $X$ by $|x-y|_X$ or $|x-y|$ when the space $X$ is
understood.  We denote by
\begin{equation*}
  B_X(x,r) = B(x,r) = \{ y : |x-y| < r \}
\end{equation*} 
the open ball of radius $r$ centered at $x$ and by $\ovl{B}(x,r)$ the
corresponding closed ball. For an open or closed ball $B$ of radius
$r$ and a number $\lambda > 0$, the notation $\lambda B$ denotes the
same type of ball with the same center and radius $\lambda r$.

We denote the complex plane with the standard Euclidean metric by
$\C$, and the disk model of hyperbolic space equipped with the
standard hyperbolic metric of curvature $-1$ by $\D$.  The open ball
of radius $\alpha>0$ centered at $0$ is denoted by $\C_\alpha$ in $\C$
and by $\D_\alpha$ in $\D$, i.e.,
\begin{equation*}
  \C_\alpha =B_\C(0,\alpha), \ \text{and}\ \D_\alpha =B_\D(0,\alpha).
\end{equation*}
Note that $\alpha$ is the Euclidean radius for $\C_\alpha$, whereas it
is the hyperbolic radius for $\D_\alpha$.

\subsection{Quasisymmetric mappings}
Let $\phi \colon X \to Y$ be an embedding of metric spaces. For $x \in
X$ and $0<r<\frac12 \diam X$, define
\begin{align*} 
  L_\phi(x,r)&:=\sup\{|\phi(x)-\phi(y)|_Y: |x-y|_X \leq r\}, \ \text{and}\\
  l_\phi(x,r)&:=\inf\{|\phi(x)-\phi(y)|_Y: |x-y|_X \geq r\}.
\end{align*}
If $\phi$ is an $\eta$-quasisymmetric embedding, then for $x \in X$
and ${0<r_1, r_2<\frac{1}{2}\diam X}$, 
\begin{equation*}
  \frac{L_\phi(x,r_1)}{l_\phi(x,r_2)} \leq
  \eta\left(\frac{r_1}{r_2}\right).
\end{equation*}

We will need a statement about the equicontinuity of quasisymmetric
maps, which is a slightly more general version of \cite[Corollary
10.27]{Heinonen2001}.
\begin{lemma}
  \label{lem:qs-embedding-equicontinuity}
  Let $\phi \colon U \to V$ be an $\eta$-quasisymmetric embedding of metric
  spaces and let $D_U$ and $D_V$ be positive real numbers.  If $\diam
  U \geq D_U$ and $\diam V \leq D_V$, then 
  \begin{equation*}
    \omega(t) = \max\left\{\frac{3D_Vt}{D_U}, D_V
      \eta\left(\frac{3t}{D_U}\right)\right\}
  \end{equation*}
  is a modulus of continuity for $\phi$.
\end{lemma}
\begin{proof}
  Given $x_1, x_2 \in U$ with $0 < |x_1 - x_2| < D_U/3$, there exists
  $x_3 \in U$ such that $|x_1-x_3| \ge D_U/3$. Then
  \begin{equation*}
    |\phi(x_1) -\phi(x_2)| \le 
    |\phi(x_1) - \phi(x_3)| \eta \left( \frac{|x_1 - x_2|}{|x_1 -
        x_3|} \right) \le D_V  \eta\left( \frac{3|x_1-x_2|}{D_U}\right).
  \end{equation*}
  If $|x_1-x_2| \geq D_U/3$, the trivial estimate $|\phi(x_1) -
  \phi(x_2)|\leq D_V$ yields the desired estimate.
\end{proof}

Note that the modulus of continuity provided by
Lemma~\ref{lem:qs-embedding-equicontinuity} is not scale-invariant. If
the metrics on $U$ and $V$ are scaled by the same quantity, the
resulting modulus of continuity may still change; see subsection
\ref{sec:data}.
 
The following result, which is a slight variation of
\cite[Theorem~2.23]{TukiaVaisala1980}, gives a local-to-global result
for quasisymmetric homeomorphisms between compact spaces in terms of
Lebesgue numbers. We include a proof for the reader's convenience.
Recall that $L > 0$ is a \emph{Lebesgue number} for a covering $\{ X_j
\}$ if for every set $E$ with $\diam E\leq L$ there exists $j$ such
that $E \subseteq X_j$.

\begin{theorem}[Tukia-V\"ais\"al\"a]\label{lem:locqs} 
  Let $F \colon X \to Y$ be a homeomorphism between compact connected
  metric spaces $X$ and $Y$. Suppose that
  \begin{itemize}
  \item $\{X_j\}_{j=1}^n$ is a finite open covering of $X$ with
    Lebesgue number $L_X < \diam X$.
  \item $\del>0$ satisfies the implication
    \begin{equation*}
      |x-x'| = L_X/2 \implies |F(x)-F(x')| \geq \del,
    \end{equation*}
  \item there is a quasisymmetric distortion function $\eta$ such that
    for each $j=1,\hdots, n$, the restricted mapping $F|_{X_j}$ is
    $\eta$-quasisymmetric.
  \end{itemize}
  Then $F$ is quasisymmetric with distortion function depending only
  on $\eta$ and the ratios $\diam(X)/L_X$ and $\diam(Y)/\del$.
\end{theorem}

\begin{proof} Let $a$, $b$, and $c$ be distinct points of $X$. Set
\begin{equation*}
  \rho = \frac{|a-b|}{|a-c|}, \ \text{and}\ \rho' =
  \frac{|F(a)-F(b)|}{|F(a)-F(c)|}.
\end{equation*}
We consider four cases.
\begin{enumerate}
\item Suppose that $\max\{|a-b|, |a-c|\} \leq L_X/2.$ Then $\{a,b,c\}$
  is contained in some $X_j$, and so $\rho' \leq \eta(\rho).$
\item Suppose that $|a-b| \leq L_X/2$ but $|a-c| > L_X/2$. Since $X$
  is connected with $\diam X >L_X$, there exists $x \in X$ with $|a-x|
  = L_X / 2$. Then there exists $j$ such that $\{ a, b, x \} \subseteq
  X_j$ and
  \begin{equation*}
    |F(a)-F(x)| \geq \del \ \text{and}\ \frac{|a-b|}{|a-x|} \leq
    \frac{2\rho\diam X}{L_X}.
  \end{equation*}
  This implies that
  \begin{equation*}
  \rho' =  \frac{|F(a)-F(b)|}{|F(a)-F(x)|}
  \frac{|F(a)-F(x)|}{|F(a)-F(c)|} \leq \eta\left(\frac{2\rho\diam
      X}{L_X}\right)\frac{\diam{Y}}{\del}.
  \end{equation*}
\item Suppose that $|a-b| > L_X/2$ but $|a-c| \leq L_X/2.$ By the same
  argument as in the previous case, there exists $x \in X$ with $|a-x|
  = L_X/2$, as well as an index $j$ such that $\{ a, c, x \} \subseteq
  X_j$. This gives
  \begin{equation*}
    |F(a)-F(x)| \geq \del \ \text{and}\ \frac{|a-x|}{|a-c|} \leq
    \frac{2\rho \diam X}{L_X}.
  \end{equation*}
  which implies that
  \begin{equation*}
    \rho' =  \frac{|F(a)-F(b)|}{|F(a)-F(x)|}
    \frac{|F(a)-F(x)|}{|F(a)-F(c)|} \leq
    \frac{\diam{Y}}{\del}\eta\left(\frac{2\rho \diam X}{L_X}\right).
  \end{equation*}
\item Suppose that $\min\{|a-b|,|a-c|\}>L_X/2.$ Then $\rho \ge
  \frac{L_X}{2 \diam X}$ and $\rho' \le \frac{\diam Y}{\delta}$,
  so
  \begin{equation*}
    \rho' \le 2 \left(\frac{\diam X}{L_X}\right) \left(\frac{\diam
        Y}{\delta}\right) \rho
  \end{equation*}

\end{enumerate}
Combining these estimates on $\rho'$ in terms of $\rho$ yields the
desired result.
\end{proof}

\subsection{Conformality and quasisymmetry} 
A conformal mapping $f \colon U \to V$ between domains in $\C$ is
quasisymmetric when restricted to a relatively compact subdomain $U'
\subset U$, with quasisymmetric distortion depending only on $U$ and
$U'$. This fact, which should be compared to Koebe's distortion
theorem, will play a key role in the proof of Theorem
\ref{thm:main}. A more general statement is even true: one may
consider quasiconformal mappings in higher dimensional Euclidean
spaces, although the quasisymmetric distortion then also depends on
the maximal quasiconformal dilatation. See \cite[Theorem
2.4]{Vaisala1981} for a proof of this fact, which we will use in the
following form:

\begin{proposition}\label{prop:conf qs}
  Let $G \colon \C_1 \to \C$ be a conformal embedding. For each $\beta \in
  (0,1)$, the restriction $G \colon \C_\beta \to \C$ is a quasisymmetric
  embedding with distortion function that depends only on $\beta$.
\end{proposition} 

An easy consequence of Proposition~\ref{prop:conf qs} is the following. 

\begin{corollary}\label{cor:choose_beta_C}
  There is a universal constant $0<\beta_{1}<1/2$ such that for any
  conformal embedding $G \colon \C_1 \to \C$,
  \begin{equation*}
    G(\C_{\beta_1}) \subeq B_{\C}\left(G(0),\frac{l_G(0,1/2)}{6}\right).
  \end{equation*}
\end{corollary} 
\begin{proof} 
  By Proposition \ref{prop:conf qs}, there is a universal
  quasisymmetric distortion function $\eta$ for $G \colon \C_{1/2} \to
  \C$. Hence, if $\beta \in (0,1/2)$, then
  \begin{equation*}
    L_G(0,\beta) \leq \eta\left(2\beta\right)l_G(0,1/2)
  \end{equation*}
  Thus, choosing $0<\beta_1<1/2$ so small that $\eta(2\beta_1)<1/6$
  fulfills the requirements of the statement.
\end{proof} 
We will also need a hyperbolic version of this result. Recall that the
hyperbolic disk $\D$ is equipped not just with a conformal structure,
but also with the standard hyperbolic metric. Our proof employs
Koebe's distortion theorem for specificity, but could also be carried
out using Proposition~\ref{prop:conf qs} alone.

\begin{proposition}\label{prop:conf qs hyp}
  Let $G \colon \C_1 \to \D$ be a conformal embedding of the Euclidean unit
  disk into the hyperbolic plane. Then the restriction
  $G \colon \C_{1/10} \to \D$ is a quasisymmetric embedding with a
  universal distortion function.
\end{proposition} 

\begin{proof} 
  The map $f(z) = \frac{G(z)-G(0)}{G'(0)} = z + O(z^2)$ is conformal
  in $\C_1$, so by Koebe's distortion theorem \cite{Duren1983},
  the image $f(\C_1)$ contains the Euclidean disk $\C_{1/4}$, and the
  image of the Euclidean disk $\C_{1/10}$ is contained in the
  Euclidean disk $\C_{10/81} \subseteq \C_{1/8}$. This implies
  \begin{equation*}
    B_\C\left(G(0), \frac{|G'(0)|}{4}\right) \subeq G(\C_1) \subeq \D,
  \end{equation*}
  and
  \begin{equation*}
    G\left(\C_{1/10}\right)\subeq B_\C\left(G(0),
      \frac{|G'(0)|}{8}\right)  \subeq B_\C\left(G(0),
      \frac{1-|G(0)|}{2}\right).
  \end{equation*}
  It follows from the definition of the hyperbolic metric that the
  identity mapping
  \begin{equation*}
    \id \colon B_\C\left(G(0), \frac{1-|G(0)|}{2}\right) \to \D
  \end{equation*}
  is, up to scaling, a bi-Lipschitz embedding with universal constants,
  so it is quasisymmetric with universal distortion. Since the
  composition of quasisymmetric maps is quasisymmetric,
  quantitatively, the result follows from Proposition \ref{prop:conf
    qs}.
\end{proof} 
  
The following statement is an easy corollary of Proposition
\ref{prop:conf qs hyp}.

\begin{corollary}\label{cor:choose beta} 
  There is a universal constant $0<\beta_2<1/10$ such that for any
  conformal embedding $G \colon \C_1 \to \D$,
  \begin{equation*}
    G(\C_{\beta_2}) \subeq B_\D\left(G(0),\frac{l_G(0,1/2)}{6}\right).
  \end{equation*}
\end{corollary}

\begin{proof} 
  By Proposition \ref{prop:conf qs hyp}, there is a universal
  quasisymmetric distortion function $\eta$ for $G \colon \C(1/10) \to
  \D$. Hence, if $\beta \in (0,1/10)$, then
  \begin{equation*}
    L_G(0,\beta) \leq \eta\left(10\beta\right)l_G(0,1/10) \leq
    \eta\left(10\beta\right)l_G(0,1/2).
  \end{equation*}
  Choosing $\beta_2$ so small that $\eta(10\beta_2)<1/6$, we get the
  claim of the corollary.
\end{proof}

For the remainder of the paper, for convenience we define
$\beta = \min\{\beta_1,\beta_2\}$.

\subsection{The data of $X$, scalings, and normalization}
\label{sec:data}

Let $X$ be a metric space that is Ahlfors $2$-regular with constant
$K$, linearly locally contractible with constant $\Lambda$, and
homeomorphic to a compact orientable surface. We will refer to $\Lambda$ and $K$
as the \emph{data} of $X$. 

For $\lambda>0$, we may form a new metric space $X_\lambda$ by
multiplying the original metric on $X$ by $\lambda$. Then $X_\lambda$
is again Ahlfors $2$-regular and linearly locally contractible, and
$X_\lambda$ has the same data as $X$. Moreover, the identity mapping
from $X$ to $X_\lambda$ is $\eta$-quasisymmetric with
$\eta(t)=t$. Hence, in the proof of Theorem~\ref{thm:main intro}, we
may scale the domain as we see fit. 

\textbf{For the remainder of this article, let $X$ be a metric space
  that is Ahlfors $2$-regular with constant $K$, linearly locally
  contractible with constant $\Lambda$, homeomorphic to a compact orientable
  surface, and has unit diameter. When we state that a quantity
  depends only the data of $X$, we are implicitly assuming this
  normalization.}

The main reason for this convention (aside from notational
convenience) is that Lemma~\ref{lem:qs-embedding-equicontinuity} is
not scale-invariant. Without this normalization, we would not be able
to say that certain moduli of continuity depend only on the data of
$X$.

\section{Finding a conformal structure} 
By \cite[Theorem 4.1]{Wildrick2010}, $X$
possesses a quasisymmetric atlas:

\begin{theorem}
  \label{thm:WildrickLocal}
  There is a quantity $A_0 \geq 1$ and a quasisymmetric distortion
  function $\eta$, each depending only on the data of $X$, such that
  for each $0<R \leq 1/ A_0$ there is a neighborhood $U$ of $x$ such
  that
  \begin{enumerate}
    \item $B(x_0,R/A_0) \subseteq U \subseteq B(x_0,A_0 R)$,
    \item there exists an $\eta$-quasisymmetric map $\phi \colon U \to \C_1$
      with $\phi(x_0) = 0$,
  \end{enumerate}
\end{theorem}
The original theorem in \cite{Wildrick2010} does not include the
normalization $\phi(x_0) = 0$. However, since $\phi \colon U \to \C_1$ is
$\eta$-quasisymmetric, the basic distortion estimates
\cite[Proposition~10.8]{Heinonen2001} imply that  $\phi(x_0) \in
\C_\alpha$ where $\alpha<1$ depends only on $\eta$. As the M\"obius
transformation $T(z) = \frac{z-\phi(x_0)}{1-\overline{\phi(x_0)}{z}}$ has
quasisymmetric distortion depending only $|\phi(x_0)|$, we may assume that $\phi(x_0)=0$. Accordingly, given a pair $(U,\phi)$
where $\phi\colon U \to \C_1$ is a homeomorphism, we will call
$\phi\inv(0)$ the \emph{center} of $(U,\phi)$.

We use the atlas provided by Theorem \ref{thm:WildrickLocal} to
produce a conformal atlas on $X$ that is adapted to its metric. We
have separated the construction into two lemmas. The first is purely
metric; the second modifies the output of the first.

\begin{lemma}
  \label{lem:qs-atlas}
  Let $\rho \in (0,1)$ be given. Then there exists a quasisymmetric
  distortion function $\eta$, radii $\alpha$, $r_0>0$, and a positive
  integer $n \in \N$, such that the following statements hold:
  \begin{enumerate}
  \item There exists an atlas $\cA_{\rho} = \{ (U_j, \C_1, \phi_j) \mid
    j=1,\ldots,n \}$ of $X$, where each mapping $\phi_j \colon U_j \to \C_1$
    is an $\eta$-quasisymmetric homeomorphism with center denoted by $x_j$.
  \item The collection $\{B(x_j, r_0)\}_{j=1}^n$ is pairwise disjoint.
  \item The collection $\{B(x_j,2r_0)\}_{j=1}^n$ covers $X$.
  \item For each $j=1,\hdots,n$, it holds that $B(x_j,10r_0) \subeq U_j$, and 
    \begin{equation*}
      \C_\alpha \subseteq \phi_j(B(x_j,r_0)) \subseteq
      \phi_j(B(x_j,10r_0)) \subeq \C_\rho.
    \end{equation*}
  \end{enumerate}
  Moreover, $\eta$ depends only on the data of $X$, while $\alpha$, $r_0$, and $n$ depend only on the data of $X$ and $\rho$.
\end{lemma}

Note that the collection $\{B(x_j,10r_0)\}_{j=1}^n$ forms an open
cover of $X$ for which $8r_0$ is a Lebesgue number; cf.\ Theorem
\ref{lem:locqs}. Moreover, Lemma~\ref{lem:qs-embedding-equicontinuity}
implies that for each $j =1,\hdots,n$, the restriction
$\phi_j|_{B(x_j,10r_0)}$ and its inverse have moduli of continuity
that depend only on $\rho$ and the data of $X$.

\begin{proof}
  For each $x \in X$, let $\phi_x \colon U_x \to \C_1$ be the
  $\eta$-quasisymmetric mapping provided by Theorem
  \ref{thm:WildrickLocal} with $R=1/A_0$, so that
  \begin{equation*}
    B(x, 1/A_0^2) \subeq U_x \subeq B(x,1/A_0).    
  \end{equation*}
  Applying Lemma~\ref{lem:qs-embedding-equicontinuity} to $\phi_x$ and
  its inverse, we see that there is a common modulus of continuity
  $\omega$ for all of the mappings $\{\phi_x,\phi_x\inv\}_{x \in X}$,
  depending only on the data of $X$.

  Hence, there is a radius $0<r_0<1/10A_0^2$ and a number $0<\alpha
  < \rho$, each depending only on $\rho$ and the data of $X$, such
  that for each $x \in X$,
  \begin{equation}\label{eq:atlas size}
    \phi_x(B(x,10r_0)) \subeq \C_\rho, \ \text{and}\ 
    \phi_x(B(x,r_0))  \supeq \C_\alpha.
  \end{equation} 
  Let $\{x_1,\hdots,x_n\}$ be a maximal $2r_0$-separated set in
  $X$. Then the open balls $\{B(x_j,r_0)\}_{j=1}^n$ are pairwise
  disjoint, while the open balls $\{B(x_j,2r_0)\}_{j=1}^n$ cover
  $X$. Since $X$ is Ahlfors $2$-regular and we have assumed that
  $\diam X =1$, this implies that $n$ is comparable to $r_0^{-2}$ and
  therefore depends only on $\rho$ and the data. Moreover, as
  $10r_0<1/(A_0)^2$, for each $j=1,\hdots,n$, it holds that
  $B(x_j,10r_0) \subeq U_{x_j}$. In particular $\{U_{x_j}\}_{j=1}^{n}$
  is also a cover of $X$.  This shows that
  $\{(U_{x_j},\C_1,\phi_{x_j})\}_{j=1}^n$ is the desired atlas.
\end{proof}

We now adapt the atlas given in Lemma \ref{lem:qs-atlas} so that the
transition mappings are conformal. This step is similar to the proof
that a quasiconformal structure on a surface has a compatible
conformal structure; see \cite{Kuusalo1967} and \cite{Cannon1969}.

We recall that $\beta \in (0,1)$ is a universal constant given by
Corollaries~\ref{cor:choose_beta_C} and~\ref{cor:choose beta}.

\begin{lemma} \label{lem:qs-conformal-atlas} 
  There exists a quasisymmetric distortion function $\eta$, radii
  $\alpha$, $r_0>0$, and a positive integer $n \in \N$, all depending
  only on the data of $X$, such that the following statements hold:
  \begin{enumerate}
  \item There exists an atlas $\cB= \{ (U_j, \C_1, \psi_j)\}_{j=1}^n$
    of $X$, where each mapping $\psi_j \colon U_j \to \C_1$ is an
    $\eta$-quasisymmetric homeomorphism with center denoted by $x_j$.
  \item The collection $\{B(x_j, r_0)\}_{j=1}^n$ is pairwise disjoint.
  \item The collection $\{B(x_j,2r_0)\}_{j=1}^n$ covers $X$.
  \item For each $j=1,\hdots,n$, it holds that $B(x_j,10r_0) \subeq U_j$, and 
    \begin{equation*}
      \C_\alpha \subseteq \psi_j(B(x_j,r_0)) \subseteq 
      \psi_j(B(x_j,10r_0)) \subeq \C_\beta.
    \end{equation*}
  \item The transition maps $\psi_{j} \circ \psi_k^{-1}$
    are conformal wherever defined.
  \end{enumerate}
\end{lemma}

\begin{proof}
  We begin by letting $0<\rho < 1$ be a number which will be
  determined below and will depend only on the data of $X$. Consider
  the atlas $\cA_{\rho}$ provided by Lemma~\ref{lem:qs-atlas}. Let us
  say that this atlas is given by $\eta'$-quasisymmetric charts
  $\{\phi_j \colon U_j \to \C_1\}_{j=1}^n$ where $\eta'$ depends only
  on the data of $X$. The associated radii $\alpha'$, $r'_0$ and the
  number of charts $n$ depend only on the data of $X$ and $\rho$.
  
  Since $X$ is connected, the charts can be relabeled to
  satisfy 
  \begin{equation*}
    U_{j+1} \cap (U_1 \cup \ldots \cup U_j) \ne \emptyset
  \end{equation*}
  for $j=1, \ldots, n-1$.

  Define $\psi_1 \colon U_1 \to \C_1$ by setting $\psi_1 = \phi_1$. We
  will iteratively construct $\psi_j \colon U_j \to \C_1$ for
  $j=2,\ldots, n$ as follows. For $k < j$, we assume that the already
  constructed map $\psi_k$ is quasisymmetric on $U_k$, has center
  $x_k$, and that if $k,k'<j$, then the transition functions $\psi_k
  \circ \psi_{k'}^{-1}$ are conformal where defined.

  In the following, we will write $D_{a,b}:=\phi_a(U_a \cap U_b)$ for
  indices $a>b$ in $\{1,\hdots,n\}$. 

  For $k<j$, define $T_{j,k} = \psi_k \circ \phi_j^{-1}$. Then
  $T_{j,k}$ is quasisymmetric and hence quasiconformal on
  $D_{j,k}$. Therefore, the complex dilatation $\mu_{j,k}$ of
  $T_{j,k}$ is well-defined (up to a.e.\ equivalence) on
  $D_{j,k}$. Given another index $k'<j$, it holds that $T_{j,k} =
  (\psi_k \circ \psi_{k'}^{-1}) \circ T_{j,k'}$ on $D_{j,k} \cap
  D_{j,k'}$. It therefore follows from the conformality of $\psi_k
  \circ \psi_{k'}^{-1}$ that $\mu_{j,k} = \mu_{j,k'}$ a.e.\ on the
  intersection $D_{j,k} \cap D_{j,k'}$. This shows that there exists a
  measurable Beltrami coefficient $\mu_j \colon \C_1 \to \C_1$ with
  $\mu_j = \mu_{j,k}$ a.~e.\ on $D_{j,k}$ for each index $1\leq k <j$,
  and $\mu_j = 0$ on $\C_1 \setminus \bigcup_{k=1}^{j-1} D_{j,k}$. We
  extend $\mu_j$ to the whole plane by $\mu_j(z) = \left(
    z/\overline{z} \right)^2 \overline{\mu_j(1/\overline{z})}$ for
  $|z|>1$. By the Measurable Riemann Mapping Theorem there exists a
  unique quasiconformal map $h_j \colon \C \to \C$, normalized by
  $h_j(0)=0$, $h_j(1) = 1$, with complex dilatation $\mu_{h_j} =
  \mu_j$ a.e.\ By the symmetry of $\mu_j$ and the chosen
  normalization, $h_j(1/\overline{z}) = 1/\overline{h_j(z)}$, and so
  $h_j (\C_1) = \C_1$. We define $\psi_j = h_j \circ \phi_j$.  The
  transformation formula for Beltrami coefficients (see, e.g.,
  \cite[IV.5.1]{LehtoVirtanen1973}) shows that if $k<j$, then $\psi_k
  \circ \psi_j\inv = T_{j,k}\circ h_j\inv$ is conformal. Moreover,
  $\psi_j(x_j) =0$.
  
  We claim that for each $j=1,\hdots,n$, the mapping $\psi_j$ has a
  quasisymmetric distortion function that depends only on the data of
  $X$; as this is true of $\phi_j$, it suffices to prove the same of
  $h_j$, and we may also assume that $j>1$. As a normalized
  quasiconformal self-map of the unit disk, $h_j$ is quasisymmetric
  with distortion controlled by the maximal dilatation
  $\|\mu_j\|_\infty$, see e.~g.\ \cite[Theorem 2.4]{Vaisala1981}. By
  an inductive argument, it is easy to show that this dilatation is
  bounded by a constant depending only on $\eta'$ (which depends only
  on the data of $X$) and the number of charts $n$ (which depends on
  the data and $\rho$). However, the bound is actually independent of
  $\rho$, by the following argument.
  
  Fix $j >1$. By the uniform quasisymmetry of $\{\phi_k\}_{k=1}^n$,
  there is a quantity $0 \leq \kappa < 1$, depending only on the data
  of $X$, such that for all indices $k< j$ with $D_{j,k} \neq
  \emptyset$, the complex dilatation $\mu_{\phi_k \circ \phi_j\inv}$
  of $\phi_k \circ \phi_j\inv \colon D_{j,k} \to \phi_k(D_{j,k})$
  satisfies $\|\mu_{\phi_k \circ \phi_j\inv}\|_{\infty} \leq \kappa.$
  For $k \in \{1,\hdots,j-1\}$, define
  \begin{equation*}
    F_{j,k} = D_{j,k} \bslash \bigcup_{l=1}^{k-1} D_{j,l}, \  F_k = \C_1 \bslash \bigcup_{l=1}^{k-1} D_{k,l}, \ 
    \text{and}\ 
    F_j = \C_1 \bslash \bigcup_{k=1}^{j-1} D_{j,k}.    
  \end{equation*}

  If $z \in F_j$, then $\mu_j(z)=0$. If $z \in F_{j,k}$, then $\phi_k
  \circ \phi_j\inv(z) \in F_k$, and hence
  \begin{equation*}
    \mu_{k}|_{\phi_k \circ \phi_j \inv (F_{j,k})} = 0.
  \end{equation*}
  The transformation formula for Beltrami coefficients now shows that
  for almost-every point $z \in F_{j,k}$,
  \begin{equation*}
    \mu_j(z) = \mu_{j,k}(z) =\mu_{\psi_k \circ \phi_j\inv}(z)= \mu_{h_k \circ \phi_k \circ \phi_j\inv}(z)  = \mu_{\phi_k \circ \phi_j \inv}(z).
  \end{equation*}

  Since $\C_1 = \bigcup_{k=1}^{j-1} F_{j,k} \cup F_j,$ we see that
  $\|\mu_j\|_\infty \leq \kappa.$ As discussed above, we may now
  conclude that each of the mappings $\{\psi_j\}_{j=1}^n$ has a
  quasisymmetric distortion function that depends only on the data of
  $X$.

  We have now seen the atlas $\cB:=\{(U_j,\C_1,\psi_j)| j =1,\hdots,n\}$
  of $X$ satisfies conditions (1) and (5) of the statement. Moreover,
  setting $r_0:=r_0'$, the conditions (2) and (3) follow directly from
  the corresponding statements for the atlas $\cA_{X,\rho}$. Note that
  only condition (4) involves the constant $\beta$. We now show how to
  choose $\rho$ so that condition (4) is satisfied.

  Let us make the convention that $h_1 \colon \C_1 \to \C_1$ is the
  identity. As discussed above, the mappings $\{h_j\}_{j=1}^n$ are
  uniformly quasisymmetric with a distortion function depending only
  on the data of $X$. Hence, by
  Lemma~\ref{lem:qs-embedding-equicontinuity}, there is a common
  modulus of continuity for all of the mappings $\{h_j,
  h_j\inv\}_{j=1}^n$ that depends only on the data of $X$.  Since
  $\beta$ is a universal constant, we may choose $\rho>0$ depending
  only on the data of $X$ such that for each $j=1,\hdots,n$, it holds
  that $h_j(\C_\rho) \subeq \C_\beta.$ Having so chosen $\rho$, the
  radius $\alpha'$ depends only on the data of $X$, and so we may also
  choose $\alpha > 0$ depending only on the data of $X$ such that for
  for each $j=1,\hdots,n$, $h_j(\C_{\alpha'}) \supeq \C_\alpha.$ This
  establishes condition (4).
\end{proof}

\subsection{Uniformizing to a standard metric}
The atlas $\mathcal{B}$ given by Lemma \ref{lem:qs-conformal-atlas}
determines a conformal structure on the compact orientable surface $X$, i.e., the
pair $(X,\mathcal{B})$ determines a Riemann surface.  By the classical
uniformization theorem, $(X,\mathcal{B})$ is conformally equivalent to
a standard Riemann surface $Y = U/\Gamma$, where $U$ denotes the
standard Riemann surface structure on the sphere, the plane, or the
unit disk, and $\Gamma$ is a discrete group of M\"obius
transformations acting freely and properly discontinuously on $U$. The
standard spherical, plane, or hyperbolic Riemannian metric on $U$ then
descends to a Riemannian metric of constant curvature $+1$, $0$, or
$-1$ on $Y$, compatible with the conformal structure. We fix a
uniformizing conformal homeomorphism $F \colon (X,\mathcal{B}) \to Y$,
and equip $Y$ with the distance function $d_Y$ arising from the
Riemannian metric of constant curvature.  Denote by $\pi$ the quotient
mapping from $U$ to $Y$.  Recall that $d_Y$ may be expressed as
$d_Y(p,q) = \inf_{\gamma} \length(\gamma),$ where the infimum is taken
over all smooth paths $\gamma$ in $U$ such that the projected path
$\pi \circ \gamma$ connects $p$ and $q$. A priori, it is not clear how
the properties of the distance $d_Y$ or the map $F$ depend on the
original metric space $(X,d)$. The following statement is the main
result of this paper, and completes the proof of Theorem~\ref{thm:main
  intro}.
\begin{theorem}
  \label{thm:main}
  The uniformizing map $F \colon (X,d) \to (Y,d_Y)$ is $\eta$-quasisymmetric,
  with distortion $\eta$ depending only on the data of $X$.
\end{theorem}

\begin{proof}[Proof of Theorem \ref{thm:main}.]
  In the case that $U=\hat{\C}$, then $X$ itself must be homeomorphic
  to $\hat{\C}$, and so Theorem~\ref{sphere} implies
  Theorem~\ref{thm:main}. We consider the remaining planar and
  hyperbolic cases together.

  Define $g_j \colon \C_1 \to F(U_j)$ by $g_j = F \circ \psi_j^{-1}$. Then
  $g_j$ is a conformal homeomorphism that lifts to a conformal
  embedding $G_j \colon \C_1 \to U$.

  We use a sequence of claims to complete the proof.

  \setcounter{claim}{0}
  \begin{claim}\label{claim:isometric projection} 
    Let $j \in \{1,\hdots,n\}$. The projection $\pi \colon U \to Y$
    maps $G_j(\C_{\beta})$ isometrically onto $g_j(\C_{\beta}).$ In
    other words, for each $z, z' \in \C_{\beta}$, $|g_i(z) -
    g_i(z')|_Y = |G_i(z)-G_i(z')|_{U}$.
  \end{claim}
  \begin{proof}[Proof of Claim~\ref{claim:isometric projection}]
    For $u \in U$, define
    \begin{equation*}
      r(u) = \min \{ |u-\gamma(u)|_{U} : \gamma \in \Gamma \setminus \{
      \id \} \}.
    \end{equation*}
    If $U=\C$, then $r(u)$ is independent of $u \in U$ and depends
    only on the group $\Gamma$. If $U=\D$, then $r(u)$ is bounded
    below by the minimal translation distance of a non-identity
    element of $\Gamma \setminus \{\id\}$, and is a $2$-Lipschitz
    function of $z$, see e.~.g.\ \cite[Section 7.35]{Beardon1983}. These facts
    imply that for any $u \in U$ and $r>0$,
    \begin{itemize}
    \item if $r \leq r(u)/6$, then $\pi|_{B_U(u,r)}$ is an isometry,
    \item if $\pi|_{B_U(u,r)}$ is injective, then $r \leq r(u)$.
    \end{itemize}
    Since $\pi \colon G_j(\C_1) \to g_j(\C_1)$ is injective, it follows
    that $l_{G_j}(0,1) \leq r(G_j(0))$.
    Corollaries~\ref{cor:choose_beta_C} and~\ref{cor:choose beta} now
    complete the proof of this claim. 
  \end{proof}
  \begin{claim}\label{claim:F_local_qs} 
    For each $j=1,\hdots,n$, the mapping $F|_{B(x_j,10r_0)}$ is
    quasisymmetric with distortion that depends only on the data of
    $X$.
  \end{claim} 
  \begin{proof}[Proof of Claim~\ref{claim:F_local_qs}]
    By Propositions~\ref{prop:conf qs} and \ref{prop:conf qs hyp}, the
    mapping $G_j$ restricted to $\C_\beta$ is quasisymmetric with a
    universal distortion function. By Claim~\ref{claim:isometric
      projection}, this is also true of $g_j$. Since we may write
    $F=g_j \circ \psi_j$, and $\psi_j$ is $\eta$-quasisymmetric where
    $\eta$ depends only on the data, it follows that for each $j
    =1,\hdots, n$, the mapping $F|_{\psi_j\inv(\C_\beta)}$ is
    quasisymmetric with distortion that depends only on the data of
    $X$. According to Lemma \ref{lem:qs-conformal-atlas},
    $\psi_j\inv(\C_\beta)\supeq B(x_j,10r_0)$, implying the claim.
  \end{proof}

  \begin{claim}\label{claim:conf_radius_bounds_hyp} 
    There is a constant $C\geq 1$ depending only on the data of $X$
    such that for each $j=1,\hdots, n$,
    \begin{equation}\label{conf_radius_bounds_hyp}
      l_{g_j}(0,\alpha) \geq \frac{\diam Y}{C}.
    \end{equation}
  \end{claim}
  \begin{proof}[Proof of Claim~\ref{claim:conf_radius_bounds_hyp}]
    As $\{B_X(x_j,2r_0)\}_{j=1}^n$ covers $X$ and
    $\psi_j(B_X(x_j,10r_0)) \subeq \C_{\beta}$, it holds that
    $\{\psi_j\inv(\C_{\beta})\}_{j=1}^n$ also covers $X$. Since $F$ is
    a homeomorphism, this implies that $Y = \bigcup_{j=1}^n
    g_j(\C_{\beta}).$ As $Y$ is connected, it follows that
    \begin{equation}\label{max} 
      \max\{\diam g_j(\C_{\beta}): j=1,\hdots,n\} \geq \frac{\diam
        Y}{n}.
    \end{equation} 
    Recall that the number of charts $n$ depends only on the data of
    $X$.

    Consider indices $j$ and $k$ in $\{1,\hdots,n\}$, and suppose that
    $|x_j-x_k|_X < 4r_0$. Then $\psi_j(x_k) \in \C_{\beta}$, and so
    $d_Y(F(x_j),F(x_k)) \leq L_{g_j}(0,\beta).$ On the other hand,
    $|x_j-x_k| \geq r_0$, and so $F(x_k) \notin g_j(\C_{\alpha})$,
    implying $d_Y(F(x_j),F(x_k)) \geq l_{g_j}(0,\alpha)$.  Moreover,
    since $g_j \colon \C_{\beta} \to Y$ is a quasisymmetric embedding
    with universal distortion function, there is a quantity $C_0 \geq
    1$ depending only on the ratio of $\beta$ to $\alpha$, and hence only
    on the data of $X$, such that
    \begin{equation*}
      l_{g_j}(0,\alpha) \geq \frac{L_{g_j}(0,\beta)}{C_0} \geq
      \frac{\diam g_j(\C_\beta)}{2C_0}
    \end{equation*}
    Since in addition,
    \begin{equation*}
      \diam g_j(\C_\beta) \geq L_{g_j}(0,\beta) \geq
      l_{g_j}(0,\alpha),
    \end{equation*}
    we have now shown that the quantities
    \begin{equation*}
      l_{g_j}(0,\alpha),\ L_{g_j}(0,\beta), \ \diam g_j(\C_\beta), \
      \text{and}\ d_Y(F(x_j),F(x_k))
    \end{equation*}
    are all comparable with constants that depend
    only on the data of $X$. The same is true with the roles of $j$
    and $k$ reversed.
    
    Since the open sets $\{B(x_j,2r_0)\}_{j=1}^n$ cover the connected
    space $X$, for any pair of indices $j,j' \in \{1,\hdots,n\}$,
    there is a sequence of distinct indices $j=j_1,\hdots, j_k=j'$ of
    length at most $n$ so that $|x_{j_i}-x_{j_{i+1}}|<4r_0$ for each
    $i=1,\hdots, k-1.$ Since $n$ depends only on the data of $X$,
    \eqref{max} proves the claim.
  \end{proof}

  We complete the proof of Theorem~\ref{thm:main} by employing
  Theorem~\ref{lem:locqs}. We consider the covering of $X$ given by
  $\{B(x_j,10r_0): j=1,\hdots, n \}$. By Claim \ref{claim:F_local_qs},
  the mapping $F$ is quasisymmetric on each element of this cover with
  a distortion function $\eta$ that depends only on the data of
  $X$. Moreover, Lemma \ref{lem:qs-conformal-atlas} implies that this
  cover has Lebesgue number $8r_0$. Suppose that $x,x' \in X$ satisfy
  $|x-x'| = 4r_0$. We may find indices $j$ and $k$ in $\{1,\hdots,n\}$
  such that $|x-x_j|< 2r_0$ and $|x'-x_k| < 2r_0$. Then $2r_0 \le
  |x'-x_j| \le 6 r_0$, so $ x, x', x_j \in B(x_j,10 r_0)$ and
  \begin{equation*}
    |F(x')-F(x)| \geq
    \frac{|F(x')-F(x_j)|}{\eta\left(\frac{|x'-x_j|}{|x'-x|}\right)} \geq
    \frac{|F(x')-F(x_j)|}{\eta\left(\frac{3}{2}\right)} 
  \end{equation*}
  Since $x' \notin B(x_j,r_0)$, we see that $F(x') \notin
  g_j(\C_{\alpha}),$ so by Claim \ref{claim:conf_radius_bounds_hyp},
  \begin{equation*}
    |F(x')-F(x_j)| = |F(x')-g_j(0)| \geq l_{g_j}(0,\alpha) \geq
    \frac{\diam Y}{C}.
  \end{equation*}
  Now Theorem \ref{lem:locqs} implies that $F \colon X \to Y$ is
  quasisymmetric with distortion function depending only on the data
  of $X$.
\end{proof}
\bibliography{everything}

\def\cprime{$'$}
\providecommand{\bysame}{\leavevmode\hbox to3em{\hrulefill}\thinspace}
\providecommand{\MR}{\relax\ifhmode\unskip\space\fi MR }
\providecommand{\MRhref}[2]{%
  \href{http://www.ams.org/mathscinet-getitem?mr=#1}{#2}
}
\providecommand{\href}[2]{#2}
\begin{thebibliography}{MW13}

\bibitem[Bea83]{Beardon1983}
Alan~F. Beardon, \emph{The geometry of discrete groups}, Graduate Texts in
  Mathematics, vol.~91, Springer-Verlag, New York, 1983. \MR{698777}

\bibitem[BK02]{BonkKleiner2002}
Mario Bonk and Bruce Kleiner, \emph{Quasisymmetric parametrizations of
  two-dimensional metric spheres}, Invent. Math. \textbf{150} (2002), no.~1,
  127--183. \MR{1930885 (2004k:53057)}

\bibitem[Bon06]{Bonk2006}
Mario Bonk, \emph{Quasiconformal geometry of fractals}, International
  {C}ongress of {M}athematicians. {V}ol. {II}, Eur. Math. Soc., Z\"urich, 2006,
  pp.~1349--1373. \MR{2275649}

\bibitem[Can69]{Cannon1969}
Raymond~J. Cannon, Jr., \emph{Quasiconformal structures and the metrization of
  {$2$}-manifolds}, Trans. Amer. Math. Soc. \textbf{135} (1969), 95--103.
  \MR{0232932 (38 \#1255)}

\bibitem[Dur83]{Duren1983}
Peter~L. Duren, \emph{Univalent functions}, Grundlehren der Mathematischen
  Wissenschaften [Fundamental Principles of Mathematical Sciences], vol. 259,
  Springer-Verlag, New York, 1983. \MR{708494}

\bibitem[Hei01]{Heinonen2001}
Juha Heinonen, \emph{Lectures on analysis on metric spaces}, Universitext,
  Springer-Verlag, New York, 2001. \MR{1800917 (2002c:30028)}

\bibitem[Kuu67]{Kuusalo1967}
Tapani Kuusalo, \emph{Verallgemeinerter {R}iemannscher {A}bbidungssatz und
  quasikonforme {M}annigfaltigkeiten}, Ann. Acad. Sci. Fenn. Ser. A I No.
  \textbf{409} (1967), 24. \MR{0218560}

\bibitem[LV73]{LehtoVirtanen1973}
O.~Lehto and K.~I. Virtanen, \emph{Quasiconformal mappings in the plane},
  second ed., Springer-Verlag, New York-Heidelberg, 1973, Translated from the
  German by K. W. Lucas, Die Grundlehren der mathematischen Wissenschaften,
  Band 126. \MR{0344463 (49 \#9202)}

\bibitem[MW13]{MerenkovWildrick2013}
Sergei Merenkov and Kevin Wildrick, \emph{Quasisymmetric {K}oebe
  uniformization}, Rev. Mat. Iberoam. \textbf{29} (2013), no.~3, 859--909.
  \MR{3090140}

\bibitem[TV80]{TukiaVaisala1980}
P.~Tukia and J.~V{\"a}is{\"a}l{\"a}, \emph{Quasisymmetric embeddings of metric
  spaces}, Ann. Acad. Sci. Fenn. Ser. A I Math. \textbf{5} (1980), no.~1,
  97--114. \MR{595180 (82g:30038)}

\bibitem[V{\"a}i81]{Vaisala1981}
Jussi V{\"a}is{\"a}l{\"a}, \emph{Quasisymmetric embeddings in {E}uclidean
  spaces}, Trans. Amer. Math. Soc. \textbf{264} (1981), no.~1, 191--204.
  \MR{597876 (82i:30031)}

\bibitem[Wil08]{Wildrick2008}
K.~Wildrick, \emph{Quasisymmetric parametrizations of two-dimensional metric
  planes}, Proc. Lond. Math. Soc. (3) \textbf{97} (2008), no.~3, 783--812.
  \MR{2448247 (2010b:30035)}

\bibitem[Wil10]{Wildrick2010}
Kevin Wildrick, \emph{Quasisymmetric structures on surfaces}, Trans. Amer.
  Math. Soc. \textbf{362} (2010), no.~2, 623--659. \MR{2551500 (2010h:30090)}

\end{thebibliography}
\end{document}